\documentclass{colt2020} 

\title[Asynchronous Stochastic Approximation and $Q$-Learning]{Finite-Time Analysis of Asynchronous Stochastic Approximation and $Q$-Learning}
\usepackage{times}
\newtheorem{assumption}{Assumption}
\newcommand{\E}{\mathbb{E}}
\newcommand{\R}{\mathbb{R}}
\newcommand{\martingale}{w}
\newcommand{\F}{\mathcal{F}}
\newcommand{\statespace}{\mathcal{N}}

\newcommand*{\qed}{\hfill\ensuremath{\blacksquare}}%

\newcommand{\operatorbound}{C}
\newcommand{\vmin}{\underline{v}}
\newcommand{\stationarydist}{{\mu}}
\newcommand{\tmix}{t_{\textsc{mix}}}
\allowdisplaybreaks

\coltauthor{%
 \Name{Guannan Qu} \Email{gqu@caltech.edu}\\
  \Name{Adam Wierman} \Email{adamw@caltech.edu}\\
 \addr Department of Computing and Mathematical Sciences\\
 California Institute of Technology\\
 Pasadena, CA 91125, USA
}

\begin{document}

\maketitle

\begin{abstract}%
  We consider a general asynchronous Stochastic Approximation (SA) scheme featuring a weighted infinity-norm contractive operator, and prove a bound on its finite-time convergence rate on a single trajectory. Additionally, we specialize the result to asynchronous $Q$-learning. The resulting bound matches the sharpest available bound for synchronous $Q$-learning, and improves over previous known bounds for asynchronous $Q$-learning.
\end{abstract}


\section{Introduction}
Reinforcement learning (RL) has received renewed interest recently due to its remarkable successes in diverse areas. 
Many RL algorithms can be viewed through the lens of Stochastic Approximation (SA) \citep{robbins1951stochastic}.  SA algorithms are widely used beyond RL in areas such as machine learning, stochastic control, signal processing, and communications and, as a result, there is a broad and deep literature focused on the analysis and applications of SA that has developed a rich class of ODE-based tools for proving convergence of SA schemes, e.g., see the books \citet{borkar2009stochastic,benveniste2012adaptive}. In the context of RL, it has been shown that linear SA captures TD-learning and that the ODE-based SA framework can be used to prove the convergence of TD-learning \citep{tsitsiklis1997analysis}. A similar connection can be found in the case of actor-critic methods 
\citep{rl_konda2000actor,konda2003linear}. 

Most of the classical analysis in SA is asymptotic in nature; however this has changed recently.  Driven by the interest in finite-time convergence of RL methods, the focus has shifted to non-asymptotic analysis of SA schemes. For example, in just the past year, a finite-time bound for linear SA is given in \citet{srikant2019finite}, which leads to finite time error bounds for TD-learning, and a finite-time bound for a linear two time scale SA model is given in \citet{gupta2019finite,doan2019finite,xu2019two}, which leads to finite-time error bounds for the gradient TD method.  These results can be viewed as extensions of the classical ODE-based SA framework, which requires the SA algorithm to admit a ``limiting'' ODE associated with a Lyapunov function that certifies stability. 

While ODE-based approaches are powerful, there are popular classes of nonlinear SA schemes featuring a nonlinear operator with infinity-norm contraction that cannot be directly analyzed from the ODE-based SA framework \citep{tsitsiklis1994asynchronous,bertsekas1996neuro}. This class of SA methods captures a particularly important class of RL methods, the Watkin's $Q$-learning method \citep{watkins1992q}, and so understanding the behavior of this class of SA schemes is important for understanding the finite-time behavior of $Q$-learning. Over the past year, progress has been made toward the finite-time analysis of these nonlinear SA schemes. In particular, \citet{shah2018q} provides  a finite-time convergence result for SA with an infinity-norm contractive operator, and \citet{wainwright2019stochastic} provides sharp convergence rates for SA with a cone-contractive operator. However, both of these works consider the \emph{synchronous} case, i.e., at each time all entries of the iterate are updated. This is a significant limitation since, in many applications, e.g., $Q$-learning on a single trajectory, the update is \emph{asynchronous}, i.e., only one of the entries is updated at a time. This leads to the following question, which is the focus of this paper: 
\begin{center}\emph{What is the finite-time convergence rate for asynchronous SA/$Q$-learning on a single trajectory?}\end{center}

\textbf{Contribution. }In this paper, we provide a finite-time analysis of asynchronous nonlinear SA schemes featuring a weighted infinity norm contraction. We prove an $O\left(\frac{1}{(1-\gamma)^{1.5}}\frac{1}{\sqrt{T}}\right)$ convergence rate in weighted infinity-norm for the SA scheme, where $\gamma$ is the contraction coefficient (Theorem~\ref{thm:main}). Notably, our results 
are sharper than the result in the synchronous case in \citet[Thm. 5]{shah2018q}.\footnote{As another related work \citet{wainwright2019stochastic} does not provide an explicit bound for the synchronous SA scheme, we can only compare with \citet{wainwright2019stochastic} in the context of $Q$-learning. }   

As a direct consequence, our result shows a $\tilde{O}(\frac{1}{(1-\gamma)^5}\frac{1}{\epsilon^2})$ convergence time to reach an $\varepsilon$-accurate (measured in infinity-norm) estimate of the $Q$-function for the asynchronous $Q$-learning method on a single trajectory in the infinite horizon $\gamma$-discounted MDP setting (Theorem~\ref{thm:q}). This result matches the sharpest known bound for synchronous $Q$-learning \citep{wainwright2019stochastic}, and to the best of our knowledge, improves over the best known finite-time bounds on asynchronous $Q$-learning \citep{even2003learning} on a single trajectory in terms of its dependence on $\frac{1}{\varepsilon}$, $\frac{1}{1-\gamma}$, and the state-action space size. Further, our results clarify a blow-up phenomenon in the asynchronous $Q$-learning literature where the error can blow up exponentially in $\frac{1}{1-\gamma}$. We show such a blow-up can be avoided by using a rescaled linear step size.  This is consistent with related findings in other settings \citep{jin2018q,wainwright2019stochastic}. 

Our proof technique is different from those in the literature, e.g.,  \citet{even2003learning,shah2018q,wainwright2019stochastic}. Specifically, we do not use an epoch-based analysis, as in \citet{even2003learning,shah2018q}, where the error is controlled epoch-by-epoch. Instead, we decompose the error in a recursive manner, and this decomposition provides a more transparent approach for analyzing how the stochastic noise impacts the approximation error. This ultimately leads to a sharper bound. Further, our approach for handling asynchronicity is very different from \citet{even2003learning} and is partially inspired by the ``drift'' analysis in the ODE-based SA literature \citet{srikant2019finite}. 

\textbf{Related Work. } 
Our results provide new insights about $Q$-learning and more generally, SA with an infinity-norm contractive operator. $Q$-learning was first proposed in \citet{watkins1992q}. Its asymptotic convergence has been proven in \citet{tsitsiklis1994asynchronous,jaakkola1994convergence}, where its connection to SA with infinity-norm contractive operator was established. The first work on non-asymptotic analysis of $Q$-learning is \citet{szepesvari1998asymptotic}, which focused on an i.i.d.\ setting.  A generalization beyond the i.i.d. setting was provided by \citet{even2003learning}, which proves finite-time bounds for synchronous and asynchronous $Q$-learning with polynomial and linear step sizes. Both \citet{szepesvari1998asymptotic} and \citet{even2003learning} discover that, when using a linear step size, there is an exponential blow-up in $\frac{1}{1-\gamma}$, where $\gamma$ is the discounting factor; further, in the asynchronous setting, there is at least cubic dependence on the state-action space size \citep[Thm. 4]{even2003learning}. Subsequently, \citet{azar2011speedy} proposes speedy $Q$-learning, a variant of synchronous $Q$-learning, by adding a momentum term, and shows it avoids the exponential blow-up with a finite time bound that scales in $\frac{1}{(1-\gamma)^4\epsilon^2}$. More recently, \citet{shah2018q,wainwright2019stochastic} provide finite time bounds for general synchronous SA, which indicates that even in the classical $Q$-learning setup, the exponential blow-up can be avoided by using a rescaled linear step size. Specifically, \citet{wainwright2019stochastic} shows a finite time bound for synchronous $Q$-learning that scales in  $\frac{1}{(1-\gamma)^5\epsilon^2}$. To the best of our knowledge, this is the sharpest known bound for synchronous $Q$-learning. Compared with the above papers, our result bridges the gap between the understanding of \emph{synchronous} SA/$Q$-learning and \emph{asynchronous} SA/$Q$-learning. Our finite time bounds for asynchronous $Q$-learning match the sharpest known scaling in $\frac{1}{(1-\gamma)}$ and $\frac{1}{\varepsilon}$ in synchronous $Q$-learning. Further, compared with the best known bounds for asynchronous $Q$-learning \citep{even2003learning}, our result improves the dependence on state-action space size from (at least) cubic to square. Additionally, our work presents a new analytic approach.

Other related work on SA and $Q$-learning include \cite{lee2019unified}, which combines the ODE-based SA framework with the switch system theory to show the asymptotic convergence of asynchronous $Q$-learning in an i.i.d. setting; \citet{beck2012error}, which studies the finite time error bound of constant step size $Q$-learning; and \citet{melo2008analysis,chen2019performance}, which analyze $Q$-learning with linear function approximation. 

We also mention that there are other lines of work on $Q$-learning focusing on different models and performance measures. One line of work seeks to propose variants of $Q$-learning, e.g. recent work \citet{wainwright2019variance} that achieves a minimax optimal rate.  Earlier examples include \citet{hasselt2010double,azar2013minimax,sidford2018near,sidford2018variance,devraj2017zap,kearns1999finite}. Compared to these papers, our work focuses on general asynchronous SA and seeks to understand the convergence of the classical form of asynchronous SA/$Q$-learning. Another related line of work on $Q$-learning focuses on proving bounds on regret, e.g.  \citet{strehl2006pac,jin2018q,dong2019q,wei2019model}. Regret is a fundamentally different goal than providing finite-time convergence bounds, and the results and techniques across the two communities are quite different. The reason is that regret bound results need to address the problem of exploration, and the performance metric focuses on the transient performance, without the need to approximate \emph{every} entry of $Q$-function to the same accuracy.  In contrast, infinity-norm finite-time error bound results typically assume a form of sufficient exploration (e.g. the i.i.d. assumption used in \citet{szepesvari1998asymptotic,lee2019unified} and the covering time assumption used in \citet{even2003learning}) and require every entry of the $Q$-function to be accurately estimated.




\section{Finite-Time Analysis of Stochastic Approximation}\label{sec:SA}
In this section, we present our results on the finite-time analysis of asynchronous SA with a (weighted) infinity-norm contractive operator. We apply the results in this section to $Q$-learning in Section~\ref{sec:q}.


To begin, we formally define the problem setting.  Let $\statespace = \{1,\ldots,n\}$, $x\in \R^{\statespace}$, and $F:\R^\statespace\rightarrow \R^\statespace$ is an operator. We use $F_i$ to denote the $i$'th entry of $F$. 
We consider the following stochastic approximation scheme that keeps updating $x(t)\in \R^\statespace$ starting from $x(0) $ being the all zero vector, 
\begin{align}
    x_{i}(t+1)& =  x_{i}(t) + \alpha_t(F_{i}(x(t)) - x_{i}(t) + \martingale(t)) &\text{ for } i=i_t, \label{eq:x_update_1}\\
    x_i(t+1) &= x_i(t) &\text{ for } i\neq i_t, \label{eq:x_update_2}
\end{align}
where $i_t\in \statespace$ is a stochastic process adapted to a filtration $\F_t$, and $w(t)$ is some noise that we will discuss later. As we show in Section~\ref{sec:q}, this stochastic approximation scheme captures the asynchronous $Q$-learning algorithm. 

Given the setting described above, the following assumptions underlie our main result. Similar to \citet{tsitsiklis1994asynchronous}, the first assumption is concerned with the contraction of $F$ in a weighted infinity norm, which we define in Definition~\ref{def:vnorm}. The reason that we consider the weighted infinity norm instead of the standard infinity norm is that its generality will capture not just the discounted case $Q$-learning, but also the undiscounted case, as shown by \citet[Sec. 7]{tsitsiklis1994asynchronous}. 

\begin{definition}[Weighted Infinity Norm] \label{def:vnorm}
Given a positive vector $v = [v_1,\ldots,v_n]^\top \in\R^{\statespace}$, the weighted infinity norm $\Vert \cdot\Vert_v$ is given by $\Vert x\Vert_v = \sup_{i\in\statespace} \frac{|x_i|}{v_i}$. 
\end{definition}

\noindent Throughout the rest of the section, we fix a positive vector $v \in \R^n $ and all the norms in the section are in $\Vert\cdot\Vert_v$. We also denote $\vmin = \inf_{i\in\statespace} v_i$, the smallest entry of $v$. We comment that when $v$ is a all one vector, $\Vert\cdot\Vert_v$ becomes the standard infinity norm. We use the following result frequently on the induced matrix norm of $\Vert\cdot\Vert_v$, the proof of which can be found in Appendix~\ref{appendix:vnorm}. 

\begin{proposition}\label{prop:vnorm}
The induced matrix norm of $\Vert\cdot\Vert_v$ for a matrix $A = [a_{ij}]_{i,j\in\statespace}$ is given by $\Vert A\Vert_v = \sup_{i\in\statespace} \sum_{j\in \statespace} \frac{v_j}{v_i} |a_{ij}|  $. When $A$ is a diagonal matrix, $\Vert A\Vert_v = \sup_{i\in\statespace} |a_{ii}|$. 
\end{proposition}

\noindent With these preparations, we are now ready to state Assumption~\ref{assump:contraction} on the contraction property of $F$.  This assumption is standard in the literature, e.g., \citep{tsitsiklis1994asynchronous,wainwright2019stochastic},\footnote{\citet{wainwright2019stochastic} considers contraction in a gauge norm associated with a cone, which is more general than the weighted infinity norm.} and is satisfied by the $Q$-learning algorithm as will be shown in Section~\ref{sec:q}. Note that, as a consequence of Assumption~\ref{assump:contraction}, $F$ has a unique fixed point $x^*$. We also note that we do not require the monotonicity assumption needed in \citet{wainwright2019stochastic}.

\begin{assumption}[Contraction]\label{assump:contraction} (a) Operator $F$ is $\gamma$ contraction in $\Vert\cdot\Vert_v$, i.e. for any $x,y\in\R^\statespace$, $\Vert F(x) - F(y)\Vert_v \leq \gamma \Vert x - y\Vert_v$. (b) There exists some constant $\operatorbound>0$ s.t. $\Vert F(x)\Vert_v \leq \gamma \Vert x\Vert_v + \operatorbound, \forall x\in\R^{\statespace}$.
\end{assumption}
Assumption~\ref{assump:contraction}(a) directly implies Assumption~\ref{assump:contraction}(b) with $C = (1+\gamma)\Vert x^*\Vert_v$.\footnote{To see this, note $\Vert F(x)\Vert_v \leq \Vert F(x) - F(x^*)\Vert_v + \Vert F(x^*)\Vert_v \leq \gamma \Vert x - x^*\Vert_v + \Vert x^*\Vert_v \leq \gamma\Vert x\Vert_v + (1+\gamma) \Vert x^*\Vert_v$.} We write Assumption~\ref{assump:contraction}(b) as a separate assumption since, in some applications (e.g. $Q$-learning), the constant $\operatorbound$ can be better than $(1+\gamma)\Vert x^*\Vert_v$. Our next assumption concerns the noise sequence $w(t)$.  It is also standard \citep{shah2018q} and is satisfied by $Q$-learning.

\begin{assumption}[Martingale Difference Sequence]\label{assump:martingale} $\martingale(t)$ is $\F_{t+1}$measurable and satisfies $\E \martingale(t) |\F_t = 0$. 
Further, $|w(t)|\leq \bar{w}$ almost surely for some constant $\bar{w}$.
\end{assumption}

\noindent Lastly, we make an assumption regarding the stochastic process $i_t$. 
\begin{assumption}[Sufficient Exploration] \label{assump:exploration}
There exists a $ \sigma\in (0,1)$ and positive integer, $\tau$, such that, for any $i\in \statespace$ and $t\geq\tau$, 
$\mathbb{P}(i_{t} = i|\mathcal{F}_{t-\tau}) \geq \sigma$.
\end{assumption}

\noindent Assumption~\ref{assump:exploration} means that, given the history up to $t-\tau$, the distribution of $i_t$ must have positive probability for every $i$. Its purpose is to ensure every $i$ is visited by $i_t$ sufficiently often. We note that Assumption~\ref{assump:exploration} is more general than many typical ergodicity assumptions used in the SA literature, e.g., \citet{srikant2019finite}. For example, the following proposition shows that if $i_t$ is an ergodic Markov chain on state space $\statespace$, then Assumption~\ref{assump:exploration} is automatically true with $\sigma$ and $\tau$ depending on the stationary distribution and the mixing time of the Markov chain, where the mixing time refers to the minimum time it takes to reach within $1/4$ total variation distance of the stationary distribution regardless of the initial state \citep[Sec. 4.5]{levin2017markov}. The proof of Proposition~\ref{prop:ergodic} can be found in Appendix~\ref{appendix:ergodic}.

\begin{proposition}\label{prop:ergodic}
If $i_t$ is a ergodic Markov chain on state space $\statespace$ with stationary distribution $\stationarydist$ and mixing time $\tmix$, then Assumption~\ref{assump:exploration} holds with $\sigma = \frac{1}{2}\stationarydist_{\min}$, where $\stationarydist_{\min} = \min_{i\in\statespace} \stationarydist_{i}$, and $\tau = \lceil \log_2(\frac{2}{\stationarydist_{\min}} )\rceil \tmix $.  
\end{proposition}

\noindent With these assumptions, we are ready to state our main result,

\begin{theorem}\label{thm:main}
Suppose Assumptions \ref{assump:contraction}, \ref{assump:martingale} and \ref{assump:exploration} hold.  Further, assume there exists constant $\bar{x}\geq \Vert x^*\Vert_v$ s.t. $\forall t, \Vert x(t)\Vert_v \leq \bar{x}$ almost surely. Let the step size be $\alpha_t = \frac{h}{t+ t_0}$ with $t_0\geq \max(4h,\tau)$, and $h\geq \frac{2}{\sigma(1-\gamma)}$.  Then, with probability at least $1-\delta$,{\small
\begin{align*}
    \Vert x(T) - x^*\Vert_v \leq \frac{12\bar{\epsilon}}{1-\gamma} \sqrt{  \frac{(\tau+1) h}{\sigma}   } \sqrt{\frac{\log(\frac{2(\tau+1) T^2 n}{\delta}) }{T+t_0}} + \frac{4}{1-\gamma}\max(\frac{16 \bar\epsilon h\tau}{\sigma}, 2\bar{x}(\tau+t_0)) \frac{1}{T+t_0},  
\end{align*}}where
$\bar{\epsilon} =  2 \bar{x} +\operatorbound+ \frac{\bar{w}}{\vmin }$.
\end{theorem}
The assumption in Theorem~\ref{thm:main} that $\Vert x(t)\Vert_v\leq \bar{x}$ is not necessary.  In particular, it can be shown (see Proposition~\ref{prop:x_bounded} below) that under Assumption~\ref{assump:contraction} and Assumption~\ref{assump:martingale}, $\Vert x(t)\Vert_v $ can be bounded by some constant almost surely. The proof of Proposition~\ref{prop:x_bounded} can be found in Appendix~\ref{appendix:x_bounded}. We treat the upper bound on $\Vert x(t)\Vert_v$ as a separate assumption because in the $Q$-learning case, the constant can be better than what is implied in Proposition~\ref{prop:x_bounded}.

\begin{proposition} \label{prop:x_bounded}
Suppose Assumptions~\ref{assump:contraction} and \ref{assump:martingale} hold.  Then for all $t$, $\Vert x(t)\Vert_v\leq \frac{1}{1-\gamma}((1+\gamma)\Vert x^*\Vert_v + \frac{\bar{w}}{\vmin}) $ almost surely.
\end{proposition}

\noindent Theorem~\ref{thm:main} shows that, when setting $h = \Theta(\frac{1}{\sigma(1-\gamma)})$ and $t_0=\Theta(\max(h,\tau))$, $\Vert x(T) - x^*\Vert_v\leq \tilde{O}( \frac{\bar{\epsilon} \sqrt{\tau}}{(1-\gamma)^{1.5} \sigma} \frac{1}{\sqrt{T}}) + \tilde{O}(\frac{\bar{\epsilon}\tau}{\sigma^2 (1-\gamma)^2}\frac{1}{T}) $. This means that, to get an approximation error of $\varepsilon$, the number of time steps required is $T \gtrsim \frac{\bar{\epsilon}^2\tau}{\sigma^2(1-\gamma)^3}\frac{1}{\varepsilon^2}$. Compared to \citet[Thm. 5]{shah2018q}, our result improves the dependence on $\frac{1}{\varepsilon}$. Note that \cite{wainwright2019stochastic} does not provide an explicit approximation bound for the SA scheme, but state the bounds in the context of $Q$-learning instead. For this reason, we compare to \citet{wainwright2019stochastic} in the context of $Q$-learning 
in Section~\ref{sec:q}.

We also comment that in the step size $\frac{h}{t+t_0}$ in Theorem~\ref{thm:main}, it is important for the $h$ constant to scale with $\Theta(\frac{1}{(1-\gamma)\sigma})$ to avoid an exponential blow-up in $\frac{1}{1-\gamma}$. This fact is not apparent in the some of the earlier work like \citet{even2003learning}, but has been pointed out recently \citep{jin2018q,wainwright2019stochastic}. Specifically, \citet{wainwright2019stochastic} shows that $h$ needs to grow with $\frac{1}{1-\gamma}$ in the synchronous SA setting. Our result is consistent with \citet{wainwright2019stochastic} and further shows that in the asynchronous setting, $h$ also needs to scale with $\frac{1}{\sigma}$. If we interpret $\sigma$ as the fraction of times that each state is visited, then such scaling in $\frac{1}{\sigma}$ will result in step size of $\Theta(\frac{1}{\sigma t})$, which is similar in spirit to a common practice in asynchronous $Q$-learning, where the step size is coordinate dependent, $\alpha_t = \Theta(\frac{1}{N_{i_t}^t})$ instead of $\Theta(\frac{1}{t})$, where $N_{i_t}^t$ means the number of times $i_t$ has been visited up to time $t$. 


\section{Application to $Q$-learning}\label{sec:q}

We now apply the results for SA to the important special case of $Q$-learning.  The setting we study is defined as follows. We consider a $\gamma$-discounted infinite horizon Markov Decision Process (MDP) with finite state space $\mathcal{S}$ and finite action space $\mathcal{A}$. Our SA result applies to both the discounted ($\gamma<1$) and undiscounted ($\gamma=1$) case. For the connection between the undiscounted case $Q$-learning and the SA scheme with the weighted infinity norm, see e.g. \citet{tsitsiklis1994asynchronous}. For ease of presentation, we focus on the discounted case ($\gamma<1$), where we can let the norm be the standard infinity norm $\Vert\cdot\Vert_\infty$, i.e., $v$ is the all-one vector.

Let the transition probability of the MDP be given by $\mathbb{P}(s_{t+1}=s'|s_t=s,a_t=a)=\mathbb{P}(s'|s,a)$. At time $t$, conditioned on the current state $s_t$ and action $a_t$, the stage reward is a random variable $r_t$ independently drawn from some fixed distribution depending on $(s_t,a_t)$, with its expectation given by $r_{s_t,a_t}$, where $r\in \R^{\mathcal{S}\times\mathcal{A}}$ is a deterministic vector.  A policy $\pi:\mathcal{S}\rightarrow \Delta(\mathcal{A}), s\mapsto \pi(\cdot|s)$ maps the state space to the probability simplex on the action space $\Delta(\mathcal{A})$, and under the policy, $a_t$ is drawn from $\pi(\cdot|s_t)$. Given a policy $\pi$, the $Q$ table $Q^\pi: \R^{\mathcal{S}\times\mathcal{A}}$ under this policy is,
$$Q^\pi_{s,a} = \E_\pi \big[\sum_{t=0}^\infty \gamma^t r_t |(s_0,a_0)=(s,a) \big],$$
where $\E_\pi$ means the expectation is taken with $a_t$ drawn from $\pi(\cdot|s_t)$. The MDP problem seeks to find an optimal policy $\pi^*$ such that $Q^{\pi}(s,a)$ is maximized simultaneously for all $(s,a)$. Classical MDP theory \citep{bertsekas1996neuro} guarantees that such a $\pi^*$ must exist and, further, the resulting $Q$-function, which we denote as $Q^*$, is the unique fixed point of the Bellman Operator  $F:\R^{\mathcal{S}\times\mathcal{A}}\rightarrow \R^{\mathcal{S}\times\mathcal{A}}$ given by, 
\begin{align}
    F_{s,a}(Q) = r_{s,a} + \gamma \E_{s'\sim \mathbb{P}(\cdot|s,a)} \max_{a'\in\mathcal{A}} Q_{s',a'}. \label{eq:bellman}
\end{align}
Once $Q^*$ is known, an optimal policy can be easily determined \citep{bertsekas1996neuro}.

When the transition probabilities and the rewards are unknown, we cannot directly use \eqref{eq:bellman} to calculate $Q^*$. The $Q$-learning algorithm is an off-policy learning algorithm that approximates $Q^*$. In the asynchronous version of $Q$-learning, we sample a trajectory $\{(s_t,a_t,r_t)\}_{t=0}^\infty$ by taking a behavioral policy $\pi$. In this process, we maintain a $Q$ table $Q(t)$, which is initialized with $Q(0)$ being the all-zero table, and is updated upon observing every new state action pair $(s_{t+1},a_{t+1})$ using the following update rule,
\begin{align}
    Q_{s_t,a_t}(t+1) &= (1-\alpha_t)Q_{s_t,a_t}(t) + \alpha_t[r_t + \gamma  \max_{a\in\mathcal{A}} Q_{s_{t+1},a}(t)  ], \label{eq:q_1}\\
    Q_{s,a} (t+1) &= Q_{s,a}(t) \text{ for }(s,a)\neq (s_t,a_t). \label{eq:q_2}
\end{align}

Our results make the following standard assumptions regarding the MDP. Assumption~\ref{assump:q}(a) is an upper bound on the reward, and Assumption~\ref{assump:q}(b) is to ensure the sufficient exploration condition in Assumption~\ref{assump:exploration} holds (cf. Proposition~\ref{prop:ergodic}).\footnote{Assumption \ref{assump:q}(b) is a simple sufficient condition that leads to Assumption~\ref{assump:exploration}, but it is not necessary. For example, Assumption~\ref{assump:exploration} does not even require the exploratory policy to be stationary. } In the asynchronous $Q$-learning literature, it is common to require some type of sufficient exploration assumption. Assumption~\ref{assump:q}(b) is more general than the i.i.d.\ assumption in \citet{szepesvari1998asymptotic,lee2019unified}, and is similar in spirit to the covering time assumption in \citet{even2003learning} and another related assumption in \citet{beck2012error}. 

\begin{assumption}\label{assump:q} The following conditions hold.
\begin{itemize}
    \item [(a)] For all $t$, the stage reward $r_t$ is upper bounded, $|r_t| \leq \bar{r}$ almost surely.
    \item [(b)] Under the behavioral policy $\pi$, the induced Markov chain with state $(s_t,a_t)$ is ergodic, has a stationary distribution $\stationarydist$ and mixing time $\tmix$. Further, define $\stationarydist_{\min} = \inf_{s,a} \stationarydist_{s,a} > 0$.
\end{itemize}
\end{assumption}

\noindent We now show that under this assumption, the $Q$-learning updates \eqref{eq:q_1} and \eqref{eq:q_2} can be written in the form of \eqref{eq:x_update_1} and \eqref{eq:x_update_2} and meet Assumptions~\ref{assump:contraction}, \ref{assump:martingale}, \ref{assump:exploration}. We first identify $\statespace = \mathcal{S}\times\mathcal{N}$, $i_t = (s_t,a_t)$, and $Q(t)$ with $x(t)$. We let $\F_t$ be the $\sigma$-algebra generated by $(s_0,a_0,r_0,\ldots,s_{t-1},a_{t-1},r_{t-1},s_t,a_t)$. Then, clearly $(s_t,a_t)$ is $\F_t$ measurable. 
We also define 
\begin{align*}
  w(t)&:=  r_t + \gamma  \max_{a\in\mathcal{A}} Q_{s_{t+1},a}(t) - F_{s_t,a_t}(Q(t))\\
  &=r_t -  r_{s_t,a_t} + \gamma \max_{a\in\mathcal{A}} Q_{s_{t+1},a}(t)  - \gamma \E_{s'\sim \mathbb{P}(\cdot|s_t,a_t)} \max_{a\in\mathcal{A}}Q_{s',a}(t).
\end{align*}
Then, \eqref{eq:q_1} can be written as,
$$ Q_{s_t,a_t}(t+1) = Q_{s_t,a_t}(t) + \alpha_t[F_{s_t,a_t}(Q(t)) + w(t) - Q_{s_t,a_t}(t)  ], $$
which shows the $Q$-learning algorithm \eqref{eq:q_1} and \eqref{eq:q_2} can be written in the form of \eqref{eq:x_update_1} and \eqref{eq:x_update_2}. We then check Assumptions~\ref{assump:contraction}, \ref{assump:martingale}, \ref{assump:exploration}. For Assumption~\ref{assump:contraction}, it is known that the Bellman Operator $F$ is a $\gamma$-contraction in infinity norm \citep{tsitsiklis1994asynchronous}; further, it easy to check $\Vert F(Q)\Vert_\infty\leq\bar{r} + \gamma \Vert Q\Vert_\infty$, and hence Assumption~\ref{assump:contraction} is met with $\operatorbound = \bar{r}$. 
For Assumption~\ref{assump:martingale}, clearly $w(t)$ is $\F_{t+1}$-measurable, and satisfies $\mathbb{E} w(t)|\F_t = 0$. For the boundedness of $w(t)$, we have the following proposition, which completes the verification of Assumption~\ref{assump:martingale}. The proof of Proposition~\ref{prop:q_ub} can be found in Appendix~\ref{appendix:q_ub}.

\begin{proposition} \label{prop:q_ub} Under Assumption~\ref{assump:q}, the $Q$-learning update satisfies the following. (a) For all $t$, $\Vert Q(t)\Vert_\infty \leq \bar{x}:= \frac{\bar{r}}{1-\gamma} $ almost surely; also, $\Vert Q^*\Vert_\infty \leq \bar{x}$. (b) For all $t$, $|w(t)| \leq \bar{w}:= \frac{2\bar{r}}{1-\gamma}$ almost surely.
\end{proposition}

\noindent Finally, using Assumption \ref{assump:q}(b) and Proposition \ref{prop:ergodic}, we have that Assumption~\ref{assump:exploration} holds with $\sigma = \frac{1}{2}\stationarydist_{\min}$ and $\tau = \lceil \log_2\frac{2}{\stationarydist_{\min}}\rceil \tmix$. 

Combining the three assumptions together with the upper bound on $\Vert Q(t)\Vert_\infty$ in Proposition~\ref{prop:q_ub}(a), we can directly apply Theorem~\ref{thm:main} and obtain the following finite-time error bounds for $Q$-learning. 

\begin{theorem}\label{thm:q}
Suppose Assumption~\ref{assump:q} holds and the step size is taken to be $\alpha_t = \frac{h}{t+ t_0}$ with $t_0\geq \max(4h,\lceil \log_2\frac{2}{\stationarydist_{\min}}\rceil \tmix)$ 
and $h\geq \frac{4}{ \stationarydist_{\min} (1-\gamma)}$. 
Then, with probability at least $1-\delta$,
{\small\begin{align*}
    \Vert Q(T) - Q^*\Vert_\infty  &\leq \frac{60\bar{r}}{(1-\gamma)^2} \sqrt{  \frac{2(\lceil \log_2\frac{2}{\stationarydist_{\min}}\rceil \tmix+1) h}{\stationarydist_{\min}}   } \sqrt{\frac{\log(\frac{2(\lceil \log_2\frac{2}{\stationarydist_{\min}}\rceil \tmix +1) T^2 |\mathcal{S}| |\mathcal{A}|}{\delta}) }{T+t_0}}\\
    &\quad + \frac{4\bar{r}}{(1-\gamma)^2}\max\Big(\frac{160  h \lceil \log_2\frac{2}{\stationarydist_{\min}}\rceil \tmix }{\stationarydist_{\min} }, 2(\lceil \log_2\frac{2}{\stationarydist_{\min}}\rceil \tmix+t_0)\Big) \frac{1}{T+t_0}.
\end{align*}}
\end{theorem}
From the above theorem, if we take $h =\Theta(\frac{1}{\stationarydist_{\min}(1-\gamma)})$, $t_0 = \tilde{\Theta} (\max(\frac{1}{\stationarydist_{\min}(1-\gamma)}, \tmix))$, the convergence rate becomes $\tilde{O}( \frac{\bar{r} \sqrt{\tmix}}{(1-\gamma)^{5/2} \stationarydist_{\min}}   \frac{1}{\sqrt{T}} +  \frac{\bar{r} \tmix}{(1-\gamma)^3\stationarydist_{\min}^2} \frac{1}{T} )$. Therefore, to reach a $\varepsilon$ accuracy in infinity norm, it takes $T \gtrsim  \frac{\bar{r}^2 \tmix}{(1-\gamma)^5 \stationarydist_{\min}^2} \frac{1}{\varepsilon^2}  $ iterations. This bound matches the best known dependence on $\frac{1}{1-\gamma}$ and $\frac{1}{\varepsilon}$ in \emph{synchronous} $Q$-learning \citep{wainwright2019stochastic}. The extra factor $\frac{\tmix}{\stationarydist_{\min}^2}$ is a result of the asynchronous updates. If we interpret $\frac{1}{\stationarydist_{\min}}$ to scale with $|\mathcal{S}|\times |\mathcal{A}|$ (the state-action space size), the extra factor becomes $\tmix (|\mathcal{S}|\times |\mathcal{A}|)^2$. We believe the scaling in $\tmix$ is inevitable. 
When compared with the results on \emph{asynchronous} $Q$-learning, to the best of our knowledge, the best finite-time bound is that of \citet[Thm. 4]{even2003learning}, where the scaling is $ \frac{(|\mathcal{S}| |\mathcal{A}|)^5}{(1-\gamma)^5 \epsilon^{2.5}}$ when $\omega = 4/5$ (optimizing dependence on $\frac{1}{1-\gamma}$), or $\frac{ (|\mathcal{S}| |\mathcal{A}|)^{3.3}}{(1-\gamma)^{5.2} \varepsilon^{2.6}} $ when $\omega = 0.77$ (optimizing dependence on $|\mathcal{S}| |\mathcal{A}| $).\footnote{Notably, \citet{even2003learning} uses a different assumption on sufficient exploration.} Here $\omega$ is a step size parameter in \citet{even2003learning}. 
While our result improves the dependence on $\frac{1}{\varepsilon}, \frac{1}{1-\gamma}, (|\mathcal{S}| |\mathcal{A}|)$ over that of \citet{even2003learning}, we believe our square dependence on the state-action space size is not optimal. We leave it as future work to investigate whether this is an intrinsic property of the algorithm or it is an artifact of the proof.

\section{Convergence Proof}\label{sec:convergence}

In this section, we prove our main result, Theorem~\ref{thm:main}. The proof is divided into three steps. In the first step, we manipulate the update equation (\eqref{eq:x_update_1} and \eqref{eq:x_update_2}) and decompose the error in a recursive form, which provides a transparent view of how the stochastic noise affects the error. In the second step, we bound the contribution of the noise sequence to the error decomposition. In the third step, we use the error decomposition and the noise sequence bounds to prove the result.

\textbf{Step 1: Decomposition of Error. } Let $\mathbf{e}_i$ to be the unit vector (the $i$'th entry is $1$ and others are zero). We let $D_t = \E \mathbf{e}_{i_t}\mathbf{e}_{i_t}^\top | \mathcal{F}_{t-\tau}$. Then, it is clear $D_t $ is a $\F_{t-\tau}$-measurable $n$-by-$n$ diagonal random matrix, with its $i$'th entry being $d_{t,i} = \mathbb{P}(i_t=i|\F_{t-\tau})$. By Assumption~\ref{assump:exploration}, we have
\begin{align}
    d_{t,i} \geq \sigma \text{ almost surely.} \label{eq:d_lowerbd}
\end{align}
With these definitions, we can rewrite the update equation \eqref{eq:x_update_1} and \eqref{eq:x_update_2} as follows,{\small
\begin{align}
    x(t+1) &= x(t) + \alpha_t[\mathbf{e}_{i_t}^\top F(x(t)) - \mathbf{e}_{i_t}^\top x(t) +  \martingale (t)]\mathbf{e}_{i_t} \nonumber \\
    &= x(t) + \alpha_t [\mathbf{e}_{i_t}\mathbf{e}_{i_t}^\top ( F(x(t)) - x(t)) + \martingale(t) \mathbf{e}_{i_t}   ]\nonumber \\
    &= x(t) + \alpha_t D_t ( F(x(t)) - x(t))   + \alpha_t [(\mathbf{e}_{i_t}\mathbf{e}_{i_t}^\top - D_t) ( F(x(t)) - x(t)) + \martingale(t) \mathbf{e}_{i_t}   ]\nonumber\\
    &= x(t) + \alpha_t [D_t F(x(t)) -D_t x(t) ]\nonumber\\
    &\quad + \alpha_t \underbrace{\big[(\mathbf{e}_{i_t}\mathbf{e}_{i_t}^\top - D_t) (F(x(t-\tau)) - x(t-\tau)) + \martingale(t) \mathbf{e}_{i_t} \big]}_{:=\epsilon(t)}\nonumber\\
    &\quad + \alpha_t \underbrace{   (\mathbf{e}_{i_t}\mathbf{e}_{i_t}^\top -D_t) [F(x(t))-F(x(t-\tau)) - (x(t) - x(t-\tau))  ]  }_{:=\phi(t)}\nonumber\\
    &= (I - \alpha_tD_t) x(t) + \alpha_t D_t F(x(t)) + \alpha_t (\epsilon(t) + \phi(t)). \label{eq:x_update_epsilon_phi}
\end{align}}Clearly, $x(t)$ is $\F_t$ measurable and $\epsilon(t)$ is $\F_{t+1}$ measurable (as $\epsilon(t)$ depends on $w(t)$, which is $\F_{t+1}$ measurable). Further,{\small
\begin{align} 
    \E \epsilon(t)|\F_{t-\tau}  
    &= \E [(\mathbf{e}_{i_t}\mathbf{e}_{i_t}^\top - D_t)| \F_{t-\tau}] [F(x(t-\tau)) - x(t-\tau)] + \E [\E [\martingale(t)| \F_t] \mathbf{e}_{i_t}|\F_{t-\tau} ] =0. \label{eq:martingale} 
\end{align}}In other words, $\epsilon(t)$ is like a ``shifted'' martingale difference sequence, where here ``shifted'' means the conditioning in \eqref{eq:martingale} is with respect to $\F_{t-\tau}$ instead of $\F_t$ as would be the case in a standard martingale difference sequence. Property \eqref{eq:martingale} will be useful later in the proof. For now, we focus on \eqref{eq:x_update_epsilon_phi} and expand it recursively, getting,{\small
\begin{align}
    x(t+1) &= \prod_{k=\tau}^t (I - \alpha_k D_k) x(\tau) + \sum_{k=\tau}^t \alpha_k D_k \prod_{\ell=k+1}^t(I - \alpha_\ell D_\ell) F(x(k)) + \sum_{k=\tau}^t \alpha_k  \prod_{\ell=k+1}^t(I - \alpha_\ell D_\ell) (\epsilon(k) + \phi(k))\nonumber  \\
    &= \tilde{B}_{\tau-1,t} x(\tau) + \sum_{k=\tau}^t B_{k,t} F(x(k))  + \sum_{k=\tau}^t \alpha_k \tilde{B}_{k,t} (\epsilon(k) + \phi(k)), \label{eq:x_sum_update}
\end{align}}where we have defined, $B_{k,t}=   \alpha_k D_k \prod_{\ell=k+1}^t(I - \alpha_\ell D_\ell)$, $\tilde{B}_{k,t} =    \prod_{\ell=k+1}^t(I - \alpha_\ell D_\ell)$. Clearly, $B_{k,t}$ and $\tilde{B}_{k,t}$ are $n$-by-$n$ diagonal random matrices, with the $i$'th diagonal entry given by $b_{k,t,i}$ and $\tilde{b}_{k,t,i}$, where $b_{k,t,i} = \alpha_k d_{k,i} \prod_{\ell=k+1}^t (1 - \alpha_\ell d_{\ell,i})$ and $\tilde{b}_{k,t,i} = \prod_{\ell=k+1}^t (1 - \alpha_\ell d_{\ell,i})$.
So, for any $i$, 
\begin{align}
    \tilde{b}_{\tau-1,t,i} + \sum_{k=\tau}^t b_{k,t,i} = 1. \label{eq:B_sum_I}
\end{align}
Also, by \eqref{eq:d_lowerbd}, we have for any $i$, almost surely
\begin{align}
    b_{k,t,i}\leq \beta_{k,t}:= \alpha_k \prod_{\ell=k+1}^{t} (1 - \alpha_\ell \sigma),  \quad 
\tilde{b}_{k,t,i}\leq \tilde{\beta}_{k,t} =  \prod_{\ell=k+1}^{t} (1 - \alpha_\ell \sigma). \label{eq:b_upperbd}
\end{align}
With these preparations, we are ready to state the following Lemma, which decomposes the error $\Vert x(t) - x^*\Vert_v$ in a recursive form. The proof of Lemma~\ref{lem:error_recursive} can be found in Appendix~\ref{subsec:error_recursive}.

\begin{lemma}\label{lem:error_recursive}
Let $a_t = \Vert x(t) - x^*\Vert_v$, we have almost surely,
\begin{align*}
 a_{t+1} \leq \tilde{\beta}_{\tau-1,t}a_\tau  + \gamma \sup_{i \in \statespace} \sum_{k=\tau}^{t}  b_{k,t,i}  a_k +  \Big\Vert \sum_{k=\tau}^{t}  \alpha_k \tilde{B}_{k,t} \epsilon(k)\Big\Vert_v  + \Big\Vert \sum_{k=\tau}^{t}  \alpha_k \tilde{B}_{k,t}  \phi(k) \Big\Vert_v.
\end{align*}
\end{lemma}
From Lemma \ref{lem:error_recursive}, it is clear that to control the error $a_t$, we need to bound $  \Vert \sum_{k=\tau}^{t}  \alpha_k \tilde{B}_{k,t} \epsilon(k) \Vert_v $ and $ \Vert \sum_{k=\tau}^{t}  \alpha_k \tilde{B}_{k,t}  \phi(k) \Vert_v$, which will be the focus of the next step.

\bigskip
\textbf{Step 2: Bounding $  \Vert \sum_{k=\tau}^{t}  \alpha_k \tilde{B}_{k,t} \epsilon(k) \Vert_v $ and $ \Vert \sum_{k=\tau}^{t}  \alpha_k \tilde{B}_{k,t}  \phi(k) \Vert_v$.} We start with a bound on each individual $\epsilon(k)$ and $\phi(k)$ in the following lemma, proven in Appendix~\ref{appendix:sec:bounded}.
\begin{lemma} \label{lem:bounded}
The following bounds hold almost surely. 
(a) $\Vert \epsilon(t)\Vert_v \leq \bar{\epsilon}:= 2 \bar{x} +\operatorbound+ \frac{\bar{w}}{\vmin }$.
(b) $\Vert \phi(t)\Vert_v \leq \sum_{k=t-\tau+1}^t 2\bar{\epsilon} \alpha_{k-1}. $ 
\end{lemma}

\noindent To bound $  \Vert \sum_{k=\tau}^{t}  \alpha_k \tilde{B}_{k,t} \epsilon(k) \Vert_v $ and $ \Vert \sum_{k=\tau}^{t}  \alpha_k \tilde{B}_{k,t}  \phi(k) \Vert_v$, we also need to understand the behavior of $\alpha_k$ and $\tilde{B}_{k,t}$. Recall that, by \eqref{eq:b_upperbd}, each entry of $B_{k,t}$ and $\tilde{B}_{k,t}$ are upper bounded by $\beta_{k,t}$ and $\tilde{\beta}_{k,t}$ respectively. We now provide the following results on the sequence $\beta_{k,t}$, $\tilde{\beta}_{k,t}$ which we will frequently use later to control $\alpha_k \tilde{B}_{k,t}$. The proof of Lemma~\ref{lem:stepsize} is provided in Appendix~\ref{appendix:sec:stepsize}.

\begin{lemma}\label{lem:stepsize}
If $\alpha_t = \frac{h}{t+t_0}$, where $ h>\frac{2}{\sigma}$ and $t_0\geq \max(4 h,\tau)$, then $\beta_{k,t},\tilde{\beta}_{k,t}$ satisfies the following. 
\begin{itemize}
    \item[(a)] $ \beta_{k,t}\leq  \frac{h}{k+t_0} \Big( \frac{k+1+t_0}{t+1+t_0}\Big)^{\sigma h}$, $ \tilde\beta_{k,t}\leq   \Big( \frac{k+1+t_0}{t+1+t_0}\Big)^{\sigma h}$.
    \item[(b)] $\sum_{k=1}^{t}\beta_{k,t}^2 \leq \frac{2h}{\sigma } \frac{1}{(t+1+t_0)} $.
    \item[(c)] $\sum_{k=\tau}^{t}\beta_{k,t}\sum_{\ell = k-\tau+1}^{k} \alpha_{\ell-1} \leq \frac{8h\tau}{\sigma} \frac{1}{t+1+t_0} $.
\end{itemize}
\end{lemma}

\noindent We are now ready to bound $  \Vert \sum_{k=\tau}^{t}  \alpha_k \tilde{B}_{k,t} \epsilon(k) \Vert_v $ and $ \Vert \sum_{k=\tau}^{t}  \alpha_k \tilde{B}_{k,t}  \phi(k) \Vert_v$. Our bound on $ \Vert \sum_{k=\tau}^{t}  \alpha_k \tilde{B}_{k,t}  \phi(k) \Vert_v$ is an immediate consequence of Lemma~\ref{lem:bounded} (b) and Lemma~\ref{lem:stepsize} (c). 

\begin{lemma}\label{lem:phi_bound} The following inequality holds almost surely, 
$$ \Big\Vert \sum_{k=\tau}^{t}  \alpha_k \tilde{B}_{k,t}  \phi(k) \Big\Vert_v \leq \frac{16 \bar\epsilon h\tau}{\sigma} \frac{1}{t+1+t_0}:=C_\phi \frac{1}{t+1+t_0}. $$
\end{lemma}

\begin{proof}
We have $\Vert \sum_{k=\tau}^{t}  \alpha_k \tilde{B}_{k,t}  \phi(k) \Vert_v\leq  \sum_{k=\tau}^{t}  \alpha_k \Vert \tilde{B}_{k,t}\Vert_v \Vert  \phi(k) \Vert_v\leq \sum_{k=\tau}^t \beta_{k,t} \sum_{\ell=k-\tau+1}^{k} 2\bar{\epsilon}\alpha_{\ell-1}\leq \frac{16\bar{\epsilon}h\tau}{\sigma(t+t_0+1)}$. Here we have used by Proposition~\ref{prop:vnorm}, $\Vert\tilde{B}_{k,t}\Vert_v = \sup_{i}|\tilde{b}_{k,t,i}|\leq \tilde{\beta}_{k,t}$.
\end{proof}


\begin{lemma}\label{lem:martingale_bd}
     For each $t$, with probability at least $1-\delta$, we have, 
     $$  \Big  \Vert \sum_{k=\tau}^{t}  \alpha_k \tilde{B}_{k,t} \epsilon(k)  \Big \Vert_v \leq  6 \bar{\epsilon} \sqrt{  \frac{(\tau+1) h}{\sigma(t+1+t_0)}   \log(\frac{2(\tau+1) t n}{\delta}) }  .   $$
\end{lemma}

\noindent We now focus on proving Lemma~\ref{lem:martingale_bd}. Recall $\epsilon(t)$ is $\F_{t+1}$ measurable is a ``shifted'' martingale difference sequence in the sense that $\E \epsilon(t)|\F_{t-\tau} = 0$ (cf. \eqref{eq:martingale}). We will use a variant of the Azuma-Hoeffding bound in Lemma~\ref{lem:azuma} that handles our ``shifted'' Martingale difference sequence. The proof of Lemma~\ref{lem:azuma} is postponed to Appendix~\ref{appendix:sec:azuma}.

\begin{lemma}\label{lem:azuma}
	Let $X_t$ be a $\mathcal{F}_t$-adapted stochastic process, satisfying
	$ \E X_t | \mathcal{F}_{t-\tau}=0$. Further, $|X_t|\leq \bar X_t$ almost surely. Then with probability $1-\delta$, we have,
	$  |\sum_{k=0}^t X_{k}  |  \leq  \sqrt{ 2\tau  \sum_{k=0}^t \bar{X}_{k}^2 \log(\frac{2\tau}{\delta}) } $. 
\end{lemma}



\noindent To prove Lemma~\ref{lem:martingale_bd}, recall that $\sum_{k=\tau}^{t} \alpha_k \tilde{B}_{k,t} \epsilon(k)$ is a random vector in $\R^{\statespace}$, with its $i$'th entry  
\begin{align}
    \sum_{k=\tau}^{t} \alpha_{k} \epsilon_i(k) \prod_{\ell=k+1}^{t} (1- \alpha_\ell d_{\ell,i}), \label{eq:martingale_entrywise}
\end{align}
with $d_{\ell,i}\geq \sigma$ almost surely, cf. \eqref{eq:d_lowerbd}. 
Fixing $i$, as have been shown in \eqref{eq:martingale}, $\epsilon_i(k)$ is a $\mathcal{F}_{k+1}$ adapted stochastic process satisfying $\E \epsilon_i(k) | \mathcal{F}_{k-\tau}=0$.  
However, $\prod_{\ell=k+1}^{t} (1- \alpha_\ell d_{\ell,i})$ is not $\mathcal{F}_{k-\tau}$-measurable, and as such we cannot directly apply the Azuma-Hoeffding bound in Lemma~\ref{lem:azuma} to  \eqref{eq:martingale_entrywise}. To proceed, we need to get rid of the randomness of $\prod_{\ell=k+1}^{t} (1- \alpha_\ell d_{\ell,i})$ in the summation \eqref{eq:martingale_entrywise}.This is done in Lemma~\ref{lem:martingale_upperbound_beta} which shows that the absolute value of quantity \eqref{eq:martingale_entrywise} can be upper bounded by the sup of another quantity where the randomness caused by $\prod_{\ell=k+1}^{t} (1- \alpha_\ell d_{\ell,i})$ is removed through the use of $d_{\ell,i}\geq \sigma$, and to this new quantity we can directly apply Lemma~\ref{lem:azuma}. The proof of Lemma~\ref{lem:martingale_upperbound_beta} is postponed to Appendix~\ref{subsec:martingale_upperbound_beta}.

\begin{lemma}\label{lem:martingale_upperbound_beta}
    For each $i$, we have almost surely,
{\small    $$\big|\sum_{k=\tau}^{t} \alpha_{k} \epsilon_i(k)  \prod_{\ell=k+1}^{t} (1- \alpha_\ell d_{\ell,i}) \big| \leq  \sup_{\tau\leq k_0\leq t}\bigg(\big|\sum_{k=k_0+1}^{t}  \epsilon_i(k)\beta_{k,t}\big| + 2 \bar{\epsilon}v_i {\beta}_{k_0,t} \bigg).$$}
\end{lemma}

\noindent With the help of Lemma~ \ref{lem:martingale_upperbound_beta}, we use the Azuma-Hoeffding bound to prove Lemma~\ref{lem:martingale_bd}.

\bigskip
\noindent\textbf{Proof of Lemma~\ref{lem:martingale_bd}.}
Fix $i$ and $\tau\leq k_0\leq t$. As have been shown in \eqref{eq:martingale}, $\frac{1}{v_i}\epsilon_i(k)\beta_{k,t}$ is a $\mathcal{F}_{k+1}$ adapted stochastic process satisfying $\E  \frac{1}{v_i}\epsilon_i(k) \beta_{k,t} | \mathcal{F}_{k-\tau}=0$.
Also by Lemma~\ref{lem:bounded}(a), $|\frac{1}{v_i}\epsilon_i(k) \beta_{k,t}|\leq \bar{\epsilon}\beta_{k,t}$ almost surely. As a result, we can use the Azuma-Hoeffding bound in Lemma~\ref{lem:azuma} to get with probability $1-\delta$,
 {\small$$ \big|\sum_{k=k_0+1}^{t}  \frac{1}{v_i} \epsilon_i(k)\beta_{k,t}\big| \leq \bar{\epsilon} \sqrt{ 2(\tau+1)  \sum_{k=k_0+1}^{t} \beta_{k,t}^2 \log(\frac{2(\tau+1)}{\delta}) } .  $$}By a union bound on $\tau\leq k_0\leq t$, we get with probability $1-\delta$,
{\small \begin{align*}
   \frac{1}{v_i} \sup_{\tau\leq k_0\leq t }\big|\sum_{k=k_0+1}^{t}  \epsilon_i(k) \beta_{k,t}\big| 
    &\leq    \bar{\epsilon} \sqrt{ 2(\tau+1)  \sum_{k=\tau+1}^{t} \beta_{k,t}^2 \log(\frac{2(\tau+1) t}{\delta}) } .
\end{align*}}Then, by Lemma~\ref{lem:martingale_upperbound_beta}, we have with probability $1-\delta$, 
{\small\begin{align*}
    \frac{1}{v_i}\Big|\sum_{k=\tau}^{t} \alpha_{k} \epsilon_i(k) \prod_{\ell=k+1}^{t} (1- \alpha_\ell d_{\ell,i})\Big| &\leq  \sup_{\tau\leq k_0\leq t}\bigg(\frac{1}{v_i}\big|\sum_{k=k_0+1}^{t}  \epsilon_i(k)\beta_{k,t}\big| + 2 \bar{\epsilon} {\beta}_{k_0,t} \bigg) \\
    &\leq \bar{\epsilon} \sqrt{ 2(\tau+1)  \sum_{k=\tau+1}^{t} \beta_{k,t}^2 \log(\frac{2(\tau+1) t}{\delta}) } +  \sup_{\tau\leq k_0\leq t}2 \bar{\epsilon} {\beta}_{k_0,t} \\
    &\leq  2 \bar{\epsilon} \sqrt{  \frac{(\tau+1) h}{\sigma(t+1+t_0)}   \log(\frac{2(\tau+1) t}{\delta}) } + \sup_{\tau\leq k_0\leq t} 2 \bar{\epsilon} \frac{h}{k_0+t_0} \Big( \frac{k_0+1+t_0}{t+1+t_0}\Big)^{\sigma h}\\
    &\leq  2 \bar{\epsilon} \sqrt{  \frac{(\tau+1) h}{\sigma(t+1+t_0)}   \log(\frac{2 (\tau+1) t}{\delta}) } + 2 \bar{\epsilon} \frac{h}{t+t_0} \\
    &\leq 6 \bar{\epsilon} \sqrt{  \frac{(\tau+1) h}{\sigma(t+1+t_0)}   \log(\frac{2 (\tau+1) t}{\delta}) }  ,
\end{align*}}where in the third inequality, we have used the bounds on $\beta_{k,t}$ in Lemma~\ref{lem:stepsize}. Finally, applying the union bound over $i\in \statespace$ will lead to the desired result. 
\qed

\textbf{Step 3: Bounding the error sequence. } We are now ready to use the error decomposition in Lemma~\ref{lem:error_recursive} and the bound on $  \Vert \sum_{k=\tau}^{t}  \alpha_k \tilde{B}_{k,t} \epsilon(k) \Vert_v $ and $ \Vert \sum_{k=\tau}^{t}  \alpha_k \tilde{B}_{k,t}  \phi(k) \Vert_v$ in Lemma~\ref{lem:martingale_bd} and Lemma~\ref{lem:phi_bound} to bound $a_t = \Vert x(t) - x^*\Vert_v$. Recall, we want to show that, with probability $1-\delta$,
\begin{align}
    a_T\leq \frac{C_a}{\sqrt{T+t_0}} + \frac{C_a'}{T+t_0},  \label{appendix:critic:eq:step5_errorbound}
\end{align}
where
$C_a =  \frac{12\bar{\epsilon}}{1-\gamma} \sqrt{  \frac{(\tau+1) h}{\sigma}   \log(\frac{2(\tau+1) T^2 n}{\delta}) }$, $C_a'= \frac{4}{1-\gamma}\max(C_\phi, 2\bar{x}(\tau+t_0)).$ To prove \eqref{appendix:critic:eq:step5_errorbound}, we start by applying Lemma~\ref{lem:martingale_bd} to $t\leq T$ with $\delta$ replaced by $\delta/T$. Then, using a union bound, we get with probability $1-\delta$, for any $t\leq T$, $\Vert \sum_{k=\tau}^{t} \alpha_k \tilde{B}_{k,t} \epsilon(k) \Vert_v \leq C_\epsilon \frac{1}{\sqrt{t+1+t_0}},$
where $ C_\epsilon = 6 \bar{\epsilon} \sqrt{  \frac{(\tau+1) h}{\sigma}   \log(\frac{2(\tau+1) T^2 n}{\delta}) }   $. Combine the above with Lemma~\ref{lem:error_recursive} and use Lemma~\ref{lem:phi_bound}, we get with probability $1-\delta$, for all $\tau\leq t\leq T$,
{\small \begin{align}
 a_{t+1} & \leq \tilde{\beta}_{\tau-1,t}a_\tau  + \gamma \sup_{i \in \statespace} \sum_{k=\tau}^{t}  b_{k,t,i}  a_k +  \Vert \sum_{k=\tau}^{t}  \alpha_k \tilde{B}_{k,t} \epsilon(k)\Vert_v  + \Vert \sum_{k=\tau}^{t}  \alpha_k \tilde{B}_{k,t}  \phi(k) \Vert_v \nonumber\\
 &\leq \tilde{\beta}_{\tau-1,t}a_\tau  + \gamma \sup_{i \in \statespace} \sum_{k=\tau}^{t}  b_{k,t,i}  a_k +  \frac{C_\epsilon}{\sqrt{t+1+t_0}}  + \frac{C_\phi}{t+1+t_0}.
    \label{eq:step_3_errorrecursive}
\end{align}
}We now condition on \eqref{eq:step_3_errorrecursive} and use induction to show \eqref{appendix:critic:eq:step5_errorbound}. Eq. \eqref{appendix:critic:eq:step5_errorbound} is true for $t=\tau$, as $\frac{C_a'}{\tau+t_0} \geq \frac{8}{1-\gamma} \bar{x} \geq a_\tau $, where we have used $a_\tau = \Vert x(\tau)-x^*\Vert_v \leq \Vert x(\tau)\Vert_v + \Vert x^*\Vert_v\leq 2\bar{x}$ by the definition of $\bar{x}$.  Then, assuming \eqref{appendix:critic:eq:step5_errorbound}  is true for up to $k\leq t$, we have by \eqref{eq:step_3_errorrecursive},
{\small\begin{align*}
    a_{t+1} &\leq \tilde{\beta}_{\tau-1,t}a_\tau  + \gamma \sup_{i\in\statespace}\sum_{k=\tau}^{t}  b_{k,t,i} [\frac{C_a}{\sqrt{k+t_0}} +\frac{C_a'}{k+t_0}] + C_\epsilon\frac{1}{\sqrt{t+1+t_0}} + C_\phi \frac{1}{t+1+t_0} \\
    &\leq \tilde{\beta}_{\tau-1,t}a_\tau  + \gamma C_a \sup_{i\in\statespace}\sum_{k=\tau}^{t}  b_{k,t,i}  \frac{1}{\sqrt{k+t_0}} + \gamma C_a' \sup_{i\in\statespace}\sum_{k=\tau}^{t}  b_{k,t,i} \frac{1}{{k+t_0}}  + C_\epsilon\frac{1}{\sqrt{t+1+t_0}} + C_\phi \frac{1}{t+1+t_0} .
\end{align*}}We use the following auxiliary Lemma, whose proof is provided in Appendix~\ref{appendix:critic:subsec:b_kt_bound}.
\begin{lemma}\label{lem:b_kt_bound}
Recall $\alpha_k = \frac{h}{k+t_0}$, and $b_{k,t,i} = \alpha_k d_{k,i} \prod_{\ell=k+1}^{t}(1-\alpha_\ell d_{\ell,i}) $, here $d_{k,i}\geq \sigma$. If $\sigma h (1-\sqrt{\gamma}) \geq 1$, $t_0\geq 1$, and $\alpha_0\leq \frac{1}{2}$, then, for any $i\in\statespace$, and any $0<\omega\leq 1$, we have
$\sum_{k=\tau}^{t}  b_{k,t,i} \frac{1}{(k+t_0)^\omega} \leq \frac{1}{\sqrt{\gamma}{(t+1+t_0)}^\omega}.  $
\end{lemma}
With Lemma~\ref{lem:b_kt_bound}, and using the bound on $\tilde{\beta}_{\tau-1,t}$ in Lemma~\ref{lem:stepsize} (a), we have 
{\small
\begin{align*}
    a_{t+1}& \leq \tilde{\beta}_{\tau-1,t}a_\tau  + \sqrt{\gamma} C_a \frac{1}{\sqrt{t+1+t_0}} + \sqrt{\gamma} C_a' \frac{1}{{t+1+t_0}} + C_\epsilon\frac{1}{\sqrt{t+1+t_0}} + C_\phi \frac{1}{t+1+t_0} \\
    &\leq  \underbrace{\sqrt{\gamma} C_a \frac{1}{\sqrt{t+1+t_0}} + C_\epsilon\frac{1}{\sqrt{t+1+t_0}}}_{:=F_t}  + \underbrace{\sqrt{\gamma} C_a' \frac{1}{{t+1+t_0}} + C_\phi \frac{1}{t+1+t_0}+ \Big(\frac{\tau+t_0}{t+1+t_0}\Big)^{\sigma h}a_\tau}_{:= F_t'}  .
\end{align*}}To finish the induction, it suffices to show $F_t\leq \frac{C_a}{\sqrt{t+1+t_0}}$ and $F_t' \leq \frac{C_a'}{t+1+t_0}$. To see this,
{\small \begin{align*}
    F_t\frac{\sqrt{t+1+t_0}}{C_a} &= \sqrt{\gamma} + \frac{C_\epsilon}{C_a}, \quad
    F_t'\frac{t+1+t_0}{C_a'} = \sqrt{\gamma} +  \frac{C_\phi}{C_a'} + \frac{a_\tau(\tau+t_0)}{C_a'} \frac{(\tau+t_0)^{\sigma h-1}}{(t+1+t_0)^{\sigma h - 1}}.
\end{align*}}It suffices to show that,
    $\frac{C_\epsilon}{C_a}  \leq  1 - \sqrt{\gamma}$,
    $\frac{C_\phi}{C_a'} \leq \frac{1 - \sqrt{\gamma}}{2}$, and
    $\frac{a_\tau {(\tau+t_0)}}{C_a'}  \leq \frac{1 - \sqrt{\gamma}}{2}$.
Using $a_\tau\leq 2\bar{x}$, one can check that $C_a$ and $C_a'$ satisfy the above three inequalities, which concludes the proof. \qed 



\bibliography{refs}

\appendix
\section{Proofs of Auxiliary Propositions in Section~\ref{sec:SA} and Section~\ref{sec:q}}
\subsection{Proof of Proposition~\ref{prop:vnorm}}\label{appendix:vnorm}
Let $x\in\R^\statespace$ be any vector s.t. $\Vert x\Vert_v = 1$. Then, 
\begin{align*}
    \Vert Ax\Vert_v = \sup_{i\in\statespace} \frac{1}{v_i} \Big|\sum_{j\in\statespace} a_{ij} x_j\Big| \leq  \sup_{i\in\statespace} \sum_{j\in\statespace}  |a_{ij}| \frac{v_j }{v_i} \frac{|x_j|}{v_j} \leq  \sup_{i\in\statespace}\sum_{j\in\statespace}  |a_{ij}| \frac{v_j }{v_i}. 
\end{align*}
As a result, $\Vert A\Vert_v\leq\sup_{i\in\statespace}\sum_{j\in\statespace}  |a_{ij}| \frac{v_j }{v_i} $. On the other hand, let $i^* = \arg\max_{i\in\statespace} \sum_{j\in\statespace}  |a_{ij}| \frac{v_j }{v_i}$ (ties broken arbitrarily). And we set $x = [x_1,\ldots,x_n]^\top$ with $x_j = v_j \text{sign}(a_{i^*j})$, where $\text{sign}(z)=1$ when $z\geq 0$, and $-1$ otherwise. Then, clearly $\Vert x\Vert_v = 1$, and 
\begin{align*}
    \Vert Ax\Vert_v \geq \frac{1}{v_{i^*}} \Big|\sum_{j\in\statespace} a_{i^*j} x_j\Big| = \frac{1}{v_{i^*}} \Big|\sum_{j\in\statespace} a_{i^*j} \text{sign}(a_{i^*j}) v_j\Big| = \sum_{j\in\statespace} |a_{i^*j}|\frac{v_j}{v_i^*} = \sup_{i\in\statespace}\sum_{j\in\statespace}  |a_{ij}| \frac{v_j }{v_i}.
\end{align*}
This shows $\Vert A\Vert_v \geq \sup_{i\in\statespace}\sum_{j\in\statespace}  |a_{ij}| \frac{v_j }{v_i}$ and finishes the proof. \qed

\subsection{Proof of Proposition~\ref{prop:ergodic}} \label{appendix:ergodic}
Let $d$ be the distribution of $i_t$ conditioned on $\mathcal{F}_{t-\tau}$. Then, by \citet[eq. (4.33)]{levin2017markov}, 
$$\text{TV}( d, \stationarydist) \leq 2^{- \lceil \log_2(\frac{2}{\stationarydist_{\min}} )\rceil } \leq  \frac{\stationarydist_{\min}}{2},$$
where $\text{TV}$ means the total-variation distance. As a result, for each $i\in\statespace$, $d_i \geq \stationarydist_i - |\stationarydist_i - d_i |\geq \stationarydist_{\min} -\text{TV}( d, \stationarydist) \geq \frac{1}{2} \stationarydist_{\min} $. This shows that for any $i$, $\mathbb{P}(i_t = i|\F_{t-\tau}) \geq \frac{1}{2}\stationarydist_{\min}$ which verifies Assumption~\ref{assump:exploration}. \qed

\subsection{Proof of Proposition~\ref{prop:x_bounded}}\label{appendix:x_bounded}

Note that by Assumption~\ref{assump:contraction}(a), we have,
$$\Vert F(x)\Vert_v \leq \Vert F(x) - F(x^*)\Vert_v+\Vert F(x^*)\Vert_v \leq \gamma \Vert x - x^*\Vert_v+ \Vert x^*\Vert_v \leq \gamma\Vert x\Vert_v+(1+\gamma)\Vert x^*\Vert_v.$$
In other words, Assumption~\ref{assump:contraction}(b) holds with $\operatorbound = (1+\gamma)\Vert x^*\Vert_v$. Let $\bar{x} = \frac{1}{1-\gamma}((1+\gamma)\Vert x^*\Vert_v + \frac{\bar{w}}{\vmin}) $. We prove $\Vert x(t)\Vert_v \leq \bar{x}$ by induction.  The statement is obviously true for $t=0$ as $x(0)$ is initialized to be the all-zero vector. Suppose it is true for $t$, then
    \begin{align*}
            \Vert x(t+1)\Vert_v &\leq \max(\frac{1}{v_{i_t}}|x_{i_t}(t+1)|, \Vert x(t)\Vert_v)\\
          & \leq \max(\frac{1}{v_{i_t}}|x_{i_t}(t+1)|, \bar{x}).
    \end{align*}
    Then, notice that,
    \begin{align*}
        \frac{1}{v_{i_t} }|x_{i_t}(t+1)| &\leq (1-\alpha_t) \frac{1}{v_{i_t}} |x_{i_t}(t)| + \alpha_t(   \frac{1}{v_{i_t} }|F_{i_t}(x(t))| +  \frac{1}{v_{i_t} } |w(t)| ) \\
        &\leq (1-\alpha_t) \Vert x(t)\Vert_v + \alpha_t( \Vert F(x(t))\Vert_v +  \frac{1}{\vmin}\bar{w} )\\
        &\leq  (1-\alpha_t) \Vert x(t)\Vert_v + \alpha_t( \gamma\Vert x(t)\Vert_v + \operatorbound + \frac{1}{\vmin} \bar{w} )\\
        &\leq (1-\alpha_t) \bar{x}+ \alpha_t( \gamma \bar{x} + \operatorbound + \bar{w} )\\
        &= \bar{x},
    \end{align*}
    where in the second inequality, we have used $|w(t)|\leq \bar{w}$ almost surely (cf. Assumption~\ref{assump:martingale}), and in the last equality, we have used that $\gamma\bar{x} + C + \frac{\bar{w}}{\vmin} = \bar{x} $. This finishes the induction. \qed

\subsection{Proof of Proposition~\ref{prop:q_ub}}\label{appendix:q_ub}
We prove $\Vert Q(t)\Vert_\infty\leq \frac{\bar{r}}{1-\gamma}$ by induction. Firstly, the statement is true for $t=0$ as $Q(0)$ is initialized to be the all zero table. Then, assume the statement is true for $t$. For $t+1$, clearly $\Vert Q(t+1)\Vert_\infty \leq \max(\Vert Q(t)\Vert_\infty, |Q_{s_t,a_t}(t+1)|)$. Further, notice,
\begin{align*}
    |Q_{s_t,a_t}(t+1)| &\leq (1-\alpha_t)|Q_{s_t,a_t}(t)| + \alpha_t(|r_t| + \gamma |\max_{a} Q_{s_{t+1},a}(t) |)\\
    &\leq (1-\alpha_t) \Vert Q(t)\Vert_\infty + \alpha_t (\bar{r} + \gamma \Vert Q(t)\Vert_\infty)\\
    &\leq (1-\alpha_t) \frac{\bar{r}}{1-\gamma} + \alpha_t(\bar{r} + \gamma \frac{\bar{r}}{1-\gamma}) \\
    &= \frac{\bar{r}}{1-\gamma}.
\end{align*}
This finishes the induction, and hence $\Vert Q(t)\Vert_\infty\leq \frac{\bar{r}}{1-\gamma}$ almost surely for all $t\geq 0$. As $Q^*$ is the $Q$-function under an optimal policy $\pi^*$, we get for any $s\in\mathcal{S},a\in\mathcal{A}$,
$$|Q^*_{s,a}| = \E_{\pi^*}[\sum_{t=0}^\infty \gamma^t r_t| (s_0,a_0) = (s,a)]\leq  \sum_{t=0}^\infty \gamma^t \bar{r} = \frac{\bar{r}}{1-\gamma},$$
which concludes the proof of part (a). For part (b), notice,
\begin{align*}
    |w(t)| &\leq |r_t| + \gamma |\max_{a} Q_{s_{t+1},a}(t) | + |F_{s_t,a_t}(Q(t))| \\
    &\leq \bar{r} + \gamma \Vert Q(t)\Vert_\infty + \Vert F(Q(t))\Vert_\infty \\
    &\leq 2 (\bar{r} + \gamma \Vert Q(t)\Vert_\infty) \\
    &\leq 2 (\bar{r} + \gamma \frac{\bar{r}}{1-\gamma})  = \frac{2\bar{r}}{1-\gamma},
\end{align*}
which finishes the proof of part (b).\qed
\section{Proofs of Auxiliary Lemmas in Section~\ref{sec:convergence}}
\subsection{Proof of Lemma~\ref{lem:error_recursive} (Error Decomposition)}\label{subsec:error_recursive}
    By \eqref{eq:x_sum_update}, we have,
    \begin{align}
       & \Vert x(t+1) - x^*\Vert_v \nonumber \\
       &\leq \sup_i  \frac{1}{v_i} \Big|  \tilde{b}_{\tau-1,t,i}x_i(\tau) + \sum_{k=\tau}^{t }  b_{k,t,i} F_i(x(k))  - x_i^*\Big|   + 
\Vert \sum_{k=\tau}^{t}  \alpha_k \tilde{B}_{k,t} \epsilon(k)\Vert_v + \Vert \sum_{k=\tau}^{t}  \alpha_k \tilde{B}_{k,t}  \phi(k)\Vert_v. \label{eq:contraction_hatq_recursive_componentwise}
    \end{align}
    Notice that by \eqref{eq:B_sum_I}, for each $i$, $\tilde{b}_{\tau-1,t,i} +\sum_{k=\tau}^{t}  b_{k,t,i} =1$. Then, for each $i$, we have 
\begin{align*}
 \frac{1}{v_i}\Big|  \tilde{b}_{\tau-1,t,i}x_i(\tau) + \sum_{k=\tau}^{t }  b_{k,t,i} F_i(x(k))  - x_i^*\Big|
&\leq \tilde{b}_{\tau-1,t,i} \frac{1}{v_i}|x_i(\tau) - x_i^*| + \sum_{k=\tau}^{t}  b_{k,t,i}   \frac{1}{v_i}\big| F_i(x(k)) - x_i^* \big|  \nonumber \\
&\leq \tilde{b}_{\tau-1,t,i} \Vert x(\tau) - x^*\Vert_v + \sum_{k=\tau}^{t}  b_{k,t,i}   \Vert F(x(k)) - x^* \Vert_v  \nonumber \\
&\leq \tilde{\beta}_{\tau-1,t}  \Vert x(\tau) - x^*\Vert_v + \gamma \sum_{k=\tau}^{t}  b_{k,t,i}   \Vert x(k) - x^* \Vert_v, 
\end{align*}
where in the last inequality, we have used that $F$ is $\gamma$-contraction in $\Vert \cdot\Vert_v$ with fixed point $x^*$. 
Combining the above with \eqref{eq:contraction_hatq_recursive_componentwise}, we have,
\begin{align*}
 & a_{t+1}=  \Vert x(t+1) -x^*\Vert_v \\
  &\leq \tilde{\beta}_{\tau-1,t}a_\tau  + \gamma \sup_{i\in \statespace} \sum_{k=\tau}^{t}  b_{k,t,i}   a_k   +  \big \Vert \sum_{k=\tau}^{t}  \alpha_k \tilde{B}_{k,t} \epsilon(k) \big\Vert_v  + \big\Vert \sum_{k=\tau}^{t}  \alpha_k \tilde{B}_{k,t}  \phi(k) \big \Vert_v.
\end{align*}
\qed

\subsection{Proof of Lemma~\ref{lem:bounded} (Bounds on $\Vert\epsilon(t)\Vert_v$ and $\Vert \phi(t)\Vert_v$)}\label{appendix:sec:bounded}

    For part (a), we have, 
    \begin{align*}
        \Vert \epsilon(t)\Vert_v &= \Vert (\mathbf{e}_{i_t}\mathbf{e}_{i_t}^\top - D_t) [F(x(t-\tau)) -x(t-\tau) ]  + \martingale(t) \mathbf{e}_{i_t} \Vert_v\\
        &\leq \Vert \mathbf{e}_{i_t}\mathbf{e}_{i_t}^\top - D_t \Vert_v \Vert F(x(t-\tau)) -x(t-\tau) )\Vert_v  + |w(t)| \Vert \mathbf{e}_{i_t} \Vert_v \\
        &\leq \Vert F(x(t-\tau)) \Vert_v + \Vert x(t-\tau)\Vert_v + \frac{\bar{w}}{\vmin }\\
        &\leq 2 \bar{x} +\operatorbound+ \frac{\bar{w}}{\vmin }  := \bar{\epsilon}.
    \end{align*}
    where we have used by Proposition~\ref{prop:vnorm}, $\Vert \mathbf{e}_{i_t}\mathbf{e}_{i_t}^\top - D_t \Vert_v = \sup_i | \mathbf{1}(i_t=i) - d_{t,i} | \leq 1$ (here $\mathbf{1}$ is the indicator function); and $\Vert F(x(t-\tau))\Vert_v \leq \gamma\Vert x(t-\tau)\Vert_v + \operatorbound \leq \bar{x} + \operatorbound$.

    For part (b), we have, 
    \begin{align*}
        \Vert \phi(t)\Vert_v &= \Vert (\mathbf{e}_{i_t}\mathbf{e}_{i_t}^\top -D_t) (F(x(t))-F(x(t-\tau))) - (\mathbf{e}_{i_t}\mathbf{e}_{i_t}^\top - D_t) (x(t) - x(t-\tau))  \Vert_v\\
        &\leq \Vert F(x(t))-F(x(t-\tau))\Vert_v +  \Vert x(t) - x(t-\tau)\Vert_v \\
        &\leq 2\Vert x(t) - x(t-\tau)\Vert_v.
    \end{align*}
    Notice that $\Vert x(t) - x(t-1)\Vert_v \leq \alpha_{t-1} (\Vert F(x(t-1))\Vert_v + \Vert x(t-1)\Vert_v + \frac{1}{\vmin} \bar{w} )\leq \alpha_{t-1} (2\bar{x} +\operatorbound+ \frac{1}{\vmin} \bar{w} )= \alpha_{t-1}\bar{\epsilon}$. Summing up, we get
    $$\Vert \phi(t)\Vert_v \leq  \sum_{k=t-\tau+1}^t 2\bar{\epsilon} \alpha_{k-1}$$\qed

\subsection{Proof of Lemma~\ref{lem:stepsize} (Step Sizes)} \label{appendix:sec:stepsize}
For part (a), notice that $\log(1-x)\leq -x$ for all $x<1$. Then, 
        $$ (1 - \sigma\alpha_t) = e^{\log( 1 - \frac{\sigma h}{t+t_0})} \leq e^{- \frac{\sigma h}{t+t_0}}. $$
Therefore,
    \begin{align*}
        \prod_{\ell=k+1}^{t}(1-\sigma \alpha_\ell)& \leq e^{-\sum_{\ell=k+1}^{t}  \frac{\sigma h}{\ell+t_0}}\\
        &\leq e^{-\int_{k+1}^{t+1}  \frac{\sigma h}{y+t_0} d y} \\
        &= e^{-\sigma h \log(\frac{t+1+t_0}{k+1+t_0})}\\
        &= \Big( \frac{k+1+t_0}{t+1+t_0}\Big)^{\sigma h},
    \end{align*}
    which leads to the bound on $\beta_{k,t}$ and $\tilde{\beta}_{k,t}$.
    
    For part (b), 
    $$\beta_{k,t}^2 \leq  \frac{h^2}{(t+1+t_0)^{2\sigma h}} \frac{(k+1+t_0)^{2\sigma h}}{(k+t_0)^2} \leq \frac{2 h^2}{(t+1+t_0)^{2\sigma h}} (k+t_0)^{2\sigma h- 2} ,$$
    where we have used $(k+1+t_0)^{2\sigma h} \leq  2 (k+t_0)^{2\sigma h}$, which is true when $t_0 \geq 4 h$. Then,
    \begin{align*}
        \sum_{k=1}^{t}\beta_{k,t}^2&\leq \frac{2 h^2}{(t+1+t_0)^{2\sigma h}}  \sum_{k=1}^{t} (k+t_0)^{2\sigma h- 2}\leq  \frac{2 h^2}{(t+1+t_0)^{2\sigma h}} \int_{1}^{t+1} (y+t_0)^{2\sigma h- 2} dy\\
    &< \frac{2 h^2}{(t+1+t_0)^{2\sigma h}} \frac{1}{2\sigma h - 1} (t+1+t_0)^{2\sigma h - 1} < \frac{2h}{\sigma } \frac{1}{t+1+t_0},
    \end{align*}
    where in the last inequality we have used $2\sigma h-1 > \sigma h$. 
    
    For part (c), notice that for $k-\tau \leq \ell\leq k-1 $ where $k\geq \tau$, we have $\alpha_\ell \leq \frac{h}{k-\tau+ t_0} \leq \frac{2h}{k+t_0}$ (using $t_0\geq \tau$). Then,
    \begin{align*}
        \sum_{k=\tau}^{t}\beta_{k,t}\sum_{\ell = k-\tau}^{k-1} \alpha_{\ell}  &\leq  \sum_{k=\tau}^{t}\beta_{k,t} \frac{2h\tau}{k+t_0}\leq \sum_{k=\tau}^{t} \frac{h}{k+t_0} \Big( \frac{k+1+t_0}{t+1+t_0}\Big)^{\sigma h} \frac{2h\tau}{k+t_0}\\
        &\leq \sum_{k=\tau}^{t} \frac{4h^2\tau}{(t+1+t_0)^{\sigma h}} (k+t_0)^{\sigma h - 2}\\
        &\leq \frac{4h^2\tau}{(t+1+t_0)^{\sigma h}} \frac{(t+1+t_0)^{\sigma h -1}}{\sigma h -1}\\
        &\leq \frac{8h\tau}{\sigma} \frac{1}{t+1+t_0},
    \end{align*}
    where we have used $(k+1+t_0)^{\sigma h}\leq 2(k+t_0)^{\sigma h}$, and $\sigma h-1 > \frac{1}{2} \sigma h$. 
\qed

\subsection{Proof of Lemma~\ref{lem:azuma} (Azuma Hoeffding)} \label{appendix:sec:azuma}
       	Let $\ell$ be an integer between $0$ and $\tau-1$. For each $\ell$, define process $Y^\ell_k = X_{\tau k + \ell}$, scalar $\bar{Y}^\ell_k = \bar{X}_{k\tau+\ell}$, and define Filtration $\tilde{\mathcal{F}}_k^\ell = \mathcal{F}_{\tau k+\ell}$. Then, $Y_k^\ell$ is $\tilde{\mathcal{F}}_k^\ell$-adapted, and satisfies
       	$$ \E Y_k^\ell | \tilde{\mathcal{F}}_{k-1}^\ell = \E X_{k\tau + \ell}| \mathcal{F}_{k\tau+\ell - \tau} = 0.$$
       	Therefore, applying Azuma-Hoeffding bound on $Y_{k}^\ell$, we have
       	$$P( |\sum_{k: k\tau + \ell\leq t} Y_{k}^\ell| \geq t ) \leq 2  \exp(-\frac{t^2}{2 \sum_{k:k\tau+\ell\leq t} (\bar{Y}^\ell_{k})^2}),$$
       	i.e. with probability at least $1 - \frac{\delta}{\tau}$,
       	$$|\sum_{k: k\tau + \ell\leq t} X_{k\tau+\ell}|=  |\sum_{k: k\tau + \ell\leq t} Y_{k}^\ell| \leq \sqrt{ 2  \sum_{k:k\tau+\ell\leq t} \bar{X}_{k\tau+\ell}^2 \log(\frac{2\tau}{\delta}) }.$$
       	Using the union bound for $\ell = 0,\ldots,\tau-1$, we get that with probability at least $1-\delta$,
       	$$|\sum_{k=0}^t X_{t}| \leq \sum_{\ell=0}^{\tau-1}|\sum_{k: k\tau + \ell\leq t} X_{k\tau+\ell}| \leq \sum_{\ell=0}^{\tau -1} \sqrt{ 2  \sum_{k:k\tau+\ell\leq t} \bar{X}_{k\tau+\ell}^2 \log(\frac{2\tau}{\delta}) }\leq  \sqrt{ 2\tau  \sum_{k=0}^t \bar{X}_{k}^2 \log(\frac{2\tau}{\delta}) },$$
    where the last inequality is due to Cauchy-Schwarz.
\qed

\subsection{Proof of Lemma~\ref{lem:martingale_upperbound_beta}} \label{subsec:martingale_upperbound_beta}

    Let $p_k$ be a scalar sequence defined as follows. Set $p_\tau=0$, and $$p_{k} = (1 - \alpha_{k-1} d_{k-1,i} )p_{k-1} + \alpha_{k-1} \epsilon_i(k-1) .$$
    Then 
    $p_{t+1} = \sum_{k=\tau}^{t} \alpha_{k} \epsilon_i(k) \prod_{\ell=k+1}^{t} (1- \alpha_\ell d_{\ell,i}) $, and to prove Lemma~\ref{lem:martingale_upperbound_beta} we need to bound $|p_{t+1}|$. 
    Let $$k_0 = \sup \{k\leq t: (1 - \alpha_{k} d_{k,i} )|p_{k}| \leq \alpha_{k} |\epsilon_i(k) |\}.$$
     We must have $k_0\geq \tau$ since $|p_\tau| = 0$. 
    With $k_0$ defined, we now define another scalar sequence $\tilde{p}$ s.t. $\tilde{p}_{k_0+1} = p_{k_0+1}$ and 
    $$\tilde{p}_{k} = (1 - \alpha_{k-1}\sigma)\tilde{p}_{k-1} + \alpha_{k-1} \epsilon_i(k-1) .$$
    We claim that for all $k\geq k_0+1$, $p_k$ and $\tilde{p}_{k}$ have the same sign, and $|p_k| \leq |\tilde p_{k}|$. This is obviously true for $k=k_0+1$. Suppose it is true for for $k-1$. Without loss of generality, suppose both $p_{k-1}$ and $\tilde{p}_{k-1}$ are non-negative. Since $k-1>k_0$ and by the definition of $k_0$, we must have
    $$ (1 - \alpha_{k-1} d_{k-1,i})p_{k-1} > |\alpha_{k-1} \epsilon_i(k-1)|. $$
    Therefore, $p_k>0$. Further, since $d_{k-1,i}\geq \sigma$, we also have
    $$(1 - \alpha_{k-1} \sigma)\tilde{p}_{k-1}\geq (1 - \alpha_{k-1} d_{k-1,i})p_{k-1} > |\alpha_{k-1} \epsilon_i(k-1)| . $$
    These imply $ \tilde{p}_k\geq p_k >0$. The case where both $p_{k-1}$ and $\tilde{p}_{k-1}$ are negative is similar. This finishes the induction, and as a result, $|p_{t+1}|\leq |\tilde{p}_{t+1}|$.
    Notice, 
    $$\tilde{p}_{t+1} = \sum_{k=k_0+1}^{t} \alpha_k \epsilon_i(k)\prod_{\ell=k+1}^{t}(1-\alpha_\ell \sigma) + \tilde{p}_{k_0+1}\prod_{\ell=k_0+1}^{t}(1-\alpha_\ell \sigma)= \sum_{k=k_0+1}^{t}  \epsilon_i(k) \beta_{k,t} + \tilde{p}_{k_0+1}\tilde{\beta}_{k_0,t}.$$
    By the definition of $k_0$, we have $$|\tilde{p}_{k_0+1}|=|p_{k_0+1}| \leq (1-\alpha_{k_0}d_{k_0,i})|p_{k_0}| + \alpha_{k_0} |\epsilon_i(k_0)| \leq 2\alpha_{k_0} |\epsilon_i(k_0)| \leq 2\alpha_{k_0} \bar{\epsilon}v_i,$$ 
 where in the last step, we have used the upper bound on $\Vert \epsilon(k_0)\Vert_v$ in Lemma~\ref{lem:bounded} (a). As a result, 
    \begin{align*}
        |p_{t+1}| & \leq |\tilde{p}_{t+1}| \leq \big|\sum_{k=k_0+1}^{t}  \epsilon_i(k)\beta_{k,t}\big| + \big|\tilde{p}_{k_0+1}\tilde{\beta}_{k_0,t}\big|\\
        &\leq \big|\sum_{k=k_0+1}^{t}  \epsilon_i(k)\beta_{k,t}\big| + \big|2 \alpha_{k_0}\bar{\epsilon}v_i \tilde{\beta}_{k_0,t}\big|\\
        &= \big|\sum_{k=k_0+1}^{t}  \epsilon_i(k)\beta_{k,t}\big| + 2 \bar{\epsilon}v_i {\beta}_{k_0,t}.
    \end{align*}
\qed

\subsection{Proof of Lemma~\ref{lem:b_kt_bound}}\label{appendix:critic:subsec:b_kt_bound}

Throughout the proof, we fix $i$ and will frequently use the property $d_{k,i}\geq \sigma$ which holds almost surely. Define the sequence 
$$e_t =\sum_{k=\tau}^{t}  b_{k,t,i} \frac{1}{{(k+t_0)^\omega}}. $$
We use induction to show that $e_t \leq \frac{1}{ \sqrt{\gamma}{(t+1+t_0)^\omega}}$. The statement is clearly true for $t=\tau$, as $e_{\tau} =  b_{\tau,\tau,i} \frac{1}{{(\tau+t_0)^\omega}}= \alpha_\tau d_{\tau,i} \frac{1}{{(\tau+t_0)^\omega}} \leq  \frac{1}{\sqrt{\gamma}(\tau+1+t_0)^\omega}$ (the last step needs $\alpha_\tau \leq \frac{1}{2}, (1+\frac{1}{t_0})^\omega\leq \frac{2}{\sqrt{\gamma}} $, implied by $t_0\geq 1$, $\omega\leq 1$). Let the statement be true for $t-1$. Then, notice that,
\begin{align*}
e_t &=\sum_{k=\tau}^{t-1}  b_{k,t,i} \frac{1}{(k+t_0)^\omega} + b_{t,t,i} \frac{1}{{(t+t_0)^\omega}} \\
&= (1-\alpha_{t}d_{t,i}) \sum_{k=\tau}^{t-1}  b_{k,t-1,i} \frac{1}{(k+t_0)^\omega} + \alpha_{t}d_{t,i} \frac{1}{{(t+t_0)^\omega}}\\
&= (1-\alpha_{t}d_{t,i})  e_{t-1}+ \alpha_{t}d_{t,i} \frac{1}{{(t+t_0)^\omega}}\\
&\leq (1-\alpha_{t}d_{t,i})  \frac{1}{\sqrt{\gamma}{(t+t_0)}^\omega} + \alpha_{t}d_{t,i}  \frac{1}{{(t+t_0)^\omega}}\\
&= \Big[ 1 - \alpha_{t} d_{t,i} (1-\sqrt{\gamma})\Big]\frac{1}{\sqrt{\gamma}{(t+t_0)}^\omega} ,
\end{align*}
where the inequality is based on induction assumption. Then, plug in $\alpha_{t} = \frac{h}{t+t_0}$ and use $d_{t,i}\geq \sigma$, we have,
\begin{align*}
e_t&\leq \Big[ 1 - \frac{\sigma h}{t+t_0} (1-\sqrt{\gamma})\Big]\frac{1}{\sqrt{\gamma}{(t+t_0)}^\omega} \\
&= \Big[ 1 - \frac{\sigma h}{t+t_0} (1-\sqrt{\gamma})\Big] \Big(\frac{t+1+t_0}{t+t_0}\Big)^\omega \frac{1}{\sqrt{\gamma}{(t+1+t_0)}^\omega}\\
&= \Big[ 1 - \frac{\sigma h}{t+t_0} (1-\sqrt{\gamma})\Big] \Big( 1+ \frac{1}{t+t_0}\Big)^\omega \frac{1}{\sqrt{\gamma}{(t+1+t_0)}^\omega}.
\end{align*}
Now using the inequality that for any $x>-1$, $(1+x)\leq e^x$, we have,
\begin{align*}
\Big[ 1 - \frac{\sigma h}{t+t_0} (1-\sqrt{\gamma})\Big] \Big( 1+ \frac{1}{t+t_0}\Big)^\omega \leq  e^{- \frac{\sigma h}{t+t_0} (1-\sqrt{\gamma}) + \omega\frac{1}{t+t_0} } \leq 1,
\end{align*}
where in the last inequality, we have used $\omega\leq 1$ and the condition on $h$ s.t. $\sigma h (1-\sqrt{\gamma}) \geq 1$. This shows $e_t \leq  \frac{1}{\sqrt{\gamma}{(t+1+t_0)}^\omega}$ and finishes the induction. \qed

\end{document}